\documentclass[11pt]{amsart}

\usepackage{amsmath, amssymb, amscd, cancel, graphicx, paralist, soul, stmaryrd}
\usepackage{tikz}
\usetikzlibrary{knots}
\usetikzlibrary{hobby}
\bibliography{refs}
\usepackage{amsthm}

\usepackage{mathdots}

\usepackage[all]{xy}
\usepackage{multirow}
\usepackage{longtable}
\usepackage{array}

\theoremstyle{plain}
\newtheorem{thm}{Theorem}
\newtheorem{conj}[thm]{Conjecture}

\newtheorem{lem}[thm]{Lemma}

\newtheorem{prop}[thm]{Proposition}

\theoremstyle{definition}
\newtheorem{defin}[thm]{Definition}
\theoremstyle{remark}
\newtheorem{rmk}[thm]{Remark}

\usepackage{hyperref}

\def\Char{{\mathrm{Char}}}

\def\disc{{\textup{disc}}}

\def\int{{\text{int}}}

\def\mod{{\textup{mod} \;}}

\def\rk{{\mathrm{rk}}}

\newcommand{\into}{\hookrightarrow}

\begin{document}

\title[Three lens space summands from the Poincar\'e homology sphere]%
{Three lens space summands from the Poincar\'e homology sphere}

\author[Jacob Caudell]{Jacob Caudell}
\begin{abstract}
    A regular fiber of the Seifert fibering of the Poincar\'e homology sphere admits a Dehn surgery to $L(2,1)\#L(3,2)\#L(5,4)$. We prove that this is the only knot in the Poincar\'e homology sphere with a surgery to a connected sum of more than two lens spaces.
\end{abstract}

\address{Department of Mathematics, Boston College\\ Chestnut Hill, MA 02467}

\email{caudell@bc.edu}
\maketitle

\section{Introduction.}

\subsection{Background.}
A classical theorem, independently proved by Lickorish \cite{Lic62} and by Wallace \cite{Wal60}, implies that for any pair of closed orientable 3-manifolds $Y$ and $Y'$, there exists a link $L\subset Y$ admitting a Dehn surgery to $Y'$. When can we characterize a knot in a given 3-manifold by the Dehn surgeries it admits?

Since Moser's classification of surgeries on torus knots in $S^3$ almost 50 years ago \cite{Mos71}, \emph{Dehn surgery characterization problems} have sustained the interest of 3-manifold topologists for decades---for example, there is the Berge conjecture \cite[Problem 1.78]{Kir10} and the cabling conjecture \cite[Conjecture A]{GAS86}, \cite[Problem 1.79]{Kir10}, among other problems. One of the most celebrated results in this direction to date is the Dehn surgery characterization of the unknot: the unknot is the only knot in $S^3$ that admits a non-trivial surgery to $S^3$ \cite{GL89}. In this paper we give a Dehn surgery characterization of a knot in the Poincar\'e homology sphere $\mathcal P$. 
\subsection{A notable surgery.}
Consider $\mathcal P$ as the Seifert fibered space $M((2,-1),(3,1),(5,1))$ \cite[Section 1.6]{Sav12}, and note that any Seifert fibering of $\mathcal P$ is isotopic to this one \cite[Corollaire 4]{BO91}. Let $F\subset \mathcal P$ be a regular fiber of $\mathcal P$, the isotopy class of which is unambiguous. Note that the two surgery diagrams in Figure 1 both present $L(2,1)\#L(3,2)\#L(5,4)$. The designated surgery slope of $0$ on $F$ specifies the slope represented by a regular fiber on the boundary of the \emph{exterior} of $F$ in $\mathcal P$, $\mathcal P_F$. We are now ready to state the main result of the current work. 
\begin{figure}
   \centering
\begin{tikzpicture}[scale=.95]
\begin{knot}[clip width = 5, flip crossing = 1, flip crossing = 4, flip crossing = 5, flip crossing = 7]
\strand[ultra thick] (0,0) to[curve through={(-2,2)..(0,4)..(2,2)}] (0,0);
\node at (0,-1.5) {\Large $-2$};
\node at (-3.5,2) {\Large $3$};
\node at (3.5,2) {\Large $5$};
\node at (-1.6,0.40) {\Large $0$};
\node[red] at (0,2.7) {\Large $0$};
\node[red] at (0,3.3) {\Large $F$};
\strand[ultra thick] (-2.25,1) to[curve through={(-3.25,2)..(-2.25,3)..(-1.25,2)}](-2.25,1);
\strand[ultra thick] (0,-1.25) to[curve through={(-1,-.25)..(0,.75)..(1,-.25)}](0,-1.25);
\strand[ultra thick] (2.25,1) to[curve through={(1.25,2)..(2.25,3)..(3.25,2)}](2.25,1);
\strand[ultra thick, red] (0,3) to[curve through = {(-1,4)..(0,5)..(1,4)}](0,3);

\node at (0+8,1) {\Large $-2$};
\node at (-2.25+8,3.25) {\Large $3$};
\node at (2.25+8,3.25) {\Large $5$};

\strand[ultra thick] (-2.25+8,1) to[curve through={(-3.25+8,2)..(-2.25+8,3)..(-1.25+8,2)}](-2.25+8,1);
\strand[ultra thick] (0+8,-1.25) to[curve through={(-1+8,-.25)..(0+8,.75)..(1+8,-.25)}](0+8,-1.25);
\strand[ultra thick] (2.25+8,1) to[curve through={(1.25+8,2)..(2.25+8,3)..(3.25+8,2)}](2.25+8,1);

\draw[dashed](4,5.5)--(4,-2);
\end{knot}
\end{tikzpicture}
    \caption{The two surgery diagrams above are related by a slam-dunk in Kirby calculus. They both present $L(2,1)\#L(3,2)\#L(5,4)$.}
    \label{fig:my_label}
\end{figure}
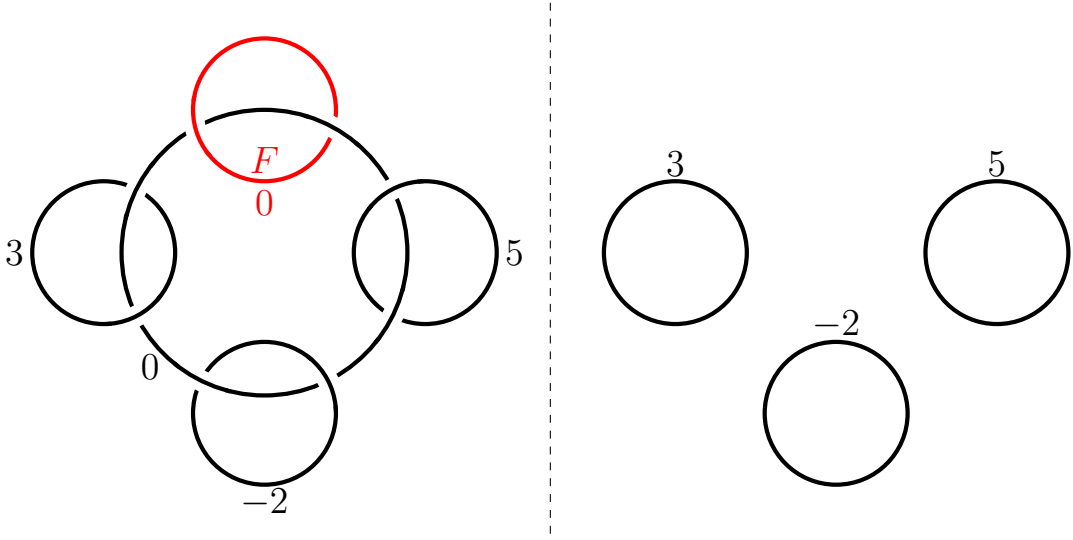

\begin{thm}
Let $K \subset \mathcal P$. Suppose $K$ admits a Dehn surgery to a connected sum of more than two lens spaces. Then $K = F$.
\end{thm}

\subsection{Changemakers reloaded.}
From the homological perspective, $\mathcal P$ is one of the simplest non-trivial 3-manifolds, as $H_*(\mathcal P) \cong H_*(S^3)$. From the cut and paste perspective, lens spaces---the 3-manifolds obtained by identifying two solid tori along their boundaries---are the simplest non-trivial 3-manifolds. From the perspective of Ozsv\'ath--Szab\'o's Heegaard Floer homology, $\mathcal P$, lens spaces, and connected sums thereof are all as simple as possible, as they realize equality in the bound $\rk\ \widehat{HF}(Y)\geq |H_1(Y)|$. We call such a 3-manifold an \emph{L-space}. We call a knot in any 3-manifold with a non-trivial surgery to an L-space an \emph{L-space knot}.

Our proof of Theorem 1 relies on the same pair of complementary genus bounds as did Greene's proof of the cabling conjecture for connected sums of lens spaces in \cite{Gre15}. In particular, we make use of the fact that $\mathcal P$ and $L(p,q)$ are L-spaces. Recall that for $K$ a knot in an integer homology 3-sphere there is a canonical identification of the set of slopes on $K$ with $\mathbb Q \cup \{1/0\}$. An appeal to the surgery exact triangle in Heegaard Floer homology gives us a favorable genus bound: for $K\subset \mathcal P$ and $p/q>0$, if \emph{$p/q$-surgery} on $K$, denoted by $K(p/q)$, is an L-space, then $p/q\geq 2g(K)-1$, where here $g(K)$ denotes the minimum genus of an orientable surface in $\mathcal P$ bounded by $K$. We first show that if $K(p/q)$ is a connected sum of more than two lens spaces, then $p/q>2g(K)-1$. With this strict inequality, we may use the work of Matignon--Sayari \cite{MS03}, building on work of Hoffman \cite{Hof98}, to find an essential 2-sphere in $K(p/q)$ that meets the core of the surgery in two points. With this 2-sphere in mind, we obtain a Seifert fibering of $\mathcal P_K$ from which we deduce Theorem 1. We show that $K(2g(K)-1)$ is never a connected sum of more than two lens spaces by way of \emph{changemaker lattice embeddings}.

\begin{defin}
Let $\{e_0,e_1,\ldots, e_n\}$ be an orthonormal basis for $-\mathbb Z^{n+1}$. A vector $\sigma = (\sigma_0,\sigma_1,\ldots, \sigma_n) \in -\mathbb Z^{n+1}$ is a \emph{changemaker} if $0\leq \sigma_0\leq \ldots\leq \sigma_n$ and for all $i \in \{1, \ldots, n\}$, $\sigma_i\leq \sum_{j=1}^{i-1}\sigma_j + 1$. A lattice $L$ is a \emph{changemaker lattice} if $L$ embeds in $-\mathbb Z^{\rk L+1}$ as the orthogonal complement to a changemaker $\sigma\in -\mathbb Z^{\rk L +1}$.
\end{defin} 

For $X$ a compact 4-manifold, denote by $Q_X$ the free $\mathbb Z$-module $H_2(X)/\text{Tors}$ equipped with the integer-valued symmetric bilinear form given by the intersection pairing of surfaces in $X$. Along the way to Theorem 1, we prove the following.

\begin{thm}
Let $K\subset \mathcal P$ be an L-space knot. If $K(2g(K)-1)$ bounds a sharp (cf. Definition 6) 4-manifold $X$ with $\rk\ {H_2(X)}=n$ and $H_1(X)$ torsion-free, then $Q_X$ embeds in $-\mathbb Z^{n+1}$ as the orthogonal complement to a changemaker $\sigma$ with $\langle \sigma, \sigma\rangle = -(2g(K)-1)$.
\end{thm}

In particular, we make use of the following variant of Donaldson's theorem originally due to Fr\o yshov \cite{Fro96} as Greene suggests in \cite[Section 1.9]{Gre13}.

\begin{thm}
Let $Z$ be a smooth, oriented, compact 4-manifold with $\rk \ Q_Z = n$ and $\partial Z = \mathcal P$, oriented as above. If $Z$ is negative definite, then $Q_Z \cong -\mathbb Z^n$ or $Q_Z\cong -E_8 \oplus -\mathbb Z^{n-8}$, where $-E_8$ is the integer lattice whose pairing is given by the matrix
\[
\begin{bmatrix}
-2 & 1 & 1 & 0 & 1 & 0 & 0 & 0\\
1 & -2 & 0 & 0 & 0 & 0 & 0 & 0\\
1 & 0 & -2 & 1 & 0 & 0 & 0 & 0\\
0 & 0 & 1 & -2 & 0 & 0 & 0 & 0\\
1 & 0 & 0 & 0 & -2 & 1 & 0 & 0\\
0 & 0 & 0 & 0 & 1 & -2 & 1 & 0\\
0 & 0 & 0 & 0 & 0 & 1 & -2 & 1\\
0 & 0 & 0 & 0 & 0 & 0 & 1 & -2
\end{bmatrix}. 
\]\qed
\end{thm}

\subsection{A generalized cabling conjecture for connected sums of lens spaces.} Contrast Theorem 1 with the abundance of knot surgeries from $S^3$ to connected sums of lens spaces. The 3-sphere admits infinitely many distinct Seifert fiberings over $S^2$ with two exceptional fibers; each torus knot appears as a regular fiber of some Seifert fibering of $S^3$. Each torus knot admits a surgery to a connected sum of lens spaces---as do appropriate cables thereof. Greene proved that torus knots and their cables account for all knots in $S^3$ with surgeries to connected sums of lens spaces. A lens space also admits infinitely many Seifert fiberings over $S^2$ with two exceptional fibers, and thus contains infinitely many knots with surgeries to a connected sum of lens spaces. Moreover, Baker showed that for any pair of integers $r$ and $s$, there are infinitely many lens spaces containing hyperbolic knots admitting surgeries to $L(r,1)\#L(s,1)$ \cite[Corollary 2]{Bak14}, subsuming all previously known examples of hyperbolic knots in lens spaces with surgeries to connected sums of lens spaces. The uniqueness of the surgery in Theorem 1 leads us to pose the following conjecture.

\begin{conj}
Let $K$ be a knot in a 3-manifold $Y$ admitting a Seifert fibering over $S^2$ with $n\geq 3$ exceptional fibers. If Dehn surgery on $K$ yields a connected sum of more than $n-1$ lens spaces, then there is a Seifert fibering of $Y$ where $K$ is a regular fiber. 
\end{conj}

In private correspondence with Baker, we learned of a conjecture more general than Conjecture 5. It posits that if a one-cusped hyperbolic 3-manifold admits two exceptional Dehn fillings with one of the fillings non-prime, then the non-prime filling has only two summands. 


\subsection{Organization.}
In Section 2 we collect the necessary input from Heegaard Floer homology and lattice theory to prove that $K(2g(K)-1)$ is never a connected sum of more than two lens spaces for $K \subset \mathcal P$. In Section 3 we invoke the main theorem of Matignon--Sayari \cite[Theorem 1.1]{MS03} to show that a knot satisfying the hypotheses of Theorem 1 is a regular fiber of $\mathcal P$.

\subsection{Conventions.}
All manifolds are assumed to be smooth and oriented. All homology groups are taken with integer coefficients. Denote by $U$ the unknot in $S^3$. We orient the lens space $L(p,q)$ as $U(-p/q)$. 

\subsection{Acknowledgments.} 
Thanks to my advisor, Josh Greene, for many insightful and enjoyable conversations, for his enduring support, and for helpful feedback on an early version of this paper. Thanks to Ken Baker for an interesting conversation about Conjecture 5.

\section{Input from Floer homology.}

\subsection{A negative definite 4-manifold with boundary $\mathcal P$.}

Recall that to a rational homology sphere $Y$ equipped with a spin$^\text{c}$ structure $\mathfrak t$, Ozsv\'ath--Szab\'o associated a numerical invariant $d(Y,\mathfrak t)\in \mathbb Q$, called a \emph{correction term}, satisfying $d(-Y,\mathfrak t)=-d(Y,\mathfrak t)$. They also proved that if $Y$ is the boundary of a negative definite 4-manifold $X$, then 
\begin{equation}
    c_1(\mathfrak s)^2 + b_2(X) \leq 4d(Y, \mathfrak t)
\end{equation}
for every $\mathfrak s \in \text{Spin}^\text{c}(X)$ extending $\mathfrak t \in \text{Spin}^\text{c}(Y)$.

\begin{defin}
A negative definite 4-manifold $X$ is \emph{sharp} if, for every $\mathfrak t \in \text{Spin}^\text{c}(Y)$, there exists some extension $\mathfrak s \in \text{Spin}^\text{c}(X)$ attaining equality in the bound $(1)$. 
\end{defin}

Following Greene's construction in \cite[Section 2]{Gre15}, let us consider a knot $K$ in an integer homology sphere $Y$ and suppose that $K(p)$ bounds a negitive definite 4-manifold $X$ with $H_1(X)$ torsion-free, where $p$ is some positive integer. Denote the trace cobordism of $p$-surgery on $K$ by $W_p(K)$, and let $W=-W_p(K)$. The homology class of the closed surface obtained by capping off a Seifert surface for $K$ with the core of the 2-handle attachment, $\Sigma$, generates $H_2(W)$ and satisfies $\langle[\Sigma],[\Sigma]\rangle = -p$. Form the oriented 4-manifold 
\begin{equation}
    Z := W\cup_{K(p)}X,
\end{equation} and note that $Z$ is negative definite with $\rk \ Q_Z=\rk\ Q_W+\rk \ Q_X$ and $\partial Z = Y$. We identify $\text{Spin}^\text{c}(K(p)) \cong \mathbb Z/p\mathbb Z$ as follows. Every $\mathfrak t\in \text{Spin}^\text{c}(K(p))$ extends to some $\mathfrak s \in \text{Spin}^\text{c}(W_p(K))$, and the residue of $\langle c_1(\mathfrak s), [\Sigma]\rangle + p$ mod $2p$ is an even integer $2i$ that is independent of the choice of extension $\mathfrak s$. The assignment $\mathfrak t \mapsto i$ gives the desired identification. With this notation in place we state the following lemma. Its proof is no different than that of \cite[Lemma 2.3]{Gre15}, which treats the case when $Y=S^3$.

\begin{lem}
Suppose that $K(p)$ bounds a smooth, negative definite 4-manifold $X$ with $\rk \ H_2(X)=n$ and $H_1(X)$ torsion-free, and form $Z = W\cup X$ as in (2). Then every $i \in \text{Spin}^\text{c}(K(p))$ extends to some $\mathfrak s \in \text{Spin}^\text{c}(Z)$, and 
\begin{equation}
    c_1(\mathfrak s)^2 + (n+1) \leq 4d(K(p),i)-4d(U(p), i).
\end{equation}
Furthermore, if $X$ is sharp, then for every $i$ there exists some extension $\mathfrak s$ that attains equality in $(3)$.\qed
\end{lem}

We now assume that $Y$ is an L-space. Let $K\subset Y$ be an L-space knot, and consider the Alexander polynomial of $K$,
\begin{equation*}
    \Delta_K(T) = \sum_{j=-g}^ga_j\cdot T^j, \ \ g := \deg(\Delta_K),
\end{equation*}
and define the \emph{torsion coefficient}
\begin{equation*}
    t_i(K) = \sum_{j\geq 1}j \cdot a_{|i|+j}.
\end{equation*}
Tange noted that the torsion coefficients of an L-space knot in any L-space integer homology sphere $Y$ with irreducible exterior have many of the same properties as they do for L-space knots in $S^3$ \cite{Tan09}. Namely, for $i\geq 0$, they form a non-increasing sequence of non-negative integers \cite{OS05}, with $t_i(K) = 0$ if and only if $i\geq g = g(K)$ \cite{OS04}. Tange also states the following result explicitly. Its proof follows from that of \cite[Corollary 7.5]{OS03}, replacing $S^3$ by an arbitrary integer homology sphere L-space.

\begin{thm}
Let $K$ be an L-space knot in an integer homology sphere L-space $Y$, and let $p$ be a positive integer. Then the torsion coefficients and correction terms satisfy
\begin{equation}
    d(Y) - 2t_i(K) = d(K(p),i)-d(U(p),i),\text{ for all } |i| \leq p/2.
\end{equation}\qed
\end{thm}

We now focus our attention to the left-hand side of (3). The first Chern class map
\begin{equation*}
    c_1: \text{Spin}^\text{c}(Z) \to H^2(Z)
\end{equation*}
has image the set of \emph{characteristic covectors} for $Q_Z$. Identify $H_2(Z) \cong H^2(Z,Y)\cong H^2(Z)$ first by Poincar\'e duality then by the long exact sequence of a pair in cohomology; then this set corresponds to
\begin{equation*}
    \Char(Q_Z) = \big\{\mathfrak c \ |\  \langle \mathfrak c, v\rangle \equiv \langle v, v \rangle\ \mod 2 \text{ for all } v \in Q_Z\big\}.
\end{equation*}

Write $\tau$ for the image of the class $[\Sigma]$ under the inclusion $H_2(W)\into H_2(Z)$. 
With the preceding notation in place, the following lemma follows on combination of Lemma 7 with Theorem 8.
\begin{lem}
Let $K$ denote an L-space knot in an L-space integer homology sphere $Y$, and suppose that $K(p)$ bounds a smooth, negative definite 4-manifold $X$ with $\rk \ H_2(X) = n$ and $H_1(X)$ torsion-free. Then
\begin{equation}
    \mathfrak c^2 + (n+1)-4d(Y)\leq -8t_i(K)
\end{equation}
for all $|i|\leq p/2$ and $\mathfrak c \in \text{Char}(Q_Z)$ such that $\langle \mathfrak c, \tau\rangle + p\equiv 2i \ \mod 2p$. Furthermore, if $X$ is sharp, then for every $|i|\leq p/2$ there exists $\mathfrak c$ attaining equality in (5).\qed
\end{lem}

Greene arrived at the notion of a changemaker lattice from a careful analysis of how Lemma 9 says elements of $\Char(Q_Z)$ of self pairing $-(n+1)$ must pair against $\tau$ when $Y= S^3$, in which case Donaldson's theorem implies $Q_Z\cong -\mathbb Z^{n+1}$. Changemaker lattices also arise for us by considering Theorem 4 in the case that a small positive integer surgery on an L-space knot in $\mathcal P$ bounds a sharp 4-manifold. 

\begin{lem}
Let $Y = \mathcal P$, and let $Z$ be as in (2). If $p \leq 2g(K)-1$, then $Q_Z\cong -\mathbb Z^{n+1}$.
\end{lem}
\begin{proof}
Since $Q_Z$ is negative definite, it follows by Theorem 4 that $Q_Z\cong -\mathbb Z^{n+1}$ or $Q_Z \cong -E_8\oplus -\mathbb Z^{n-7}$. Suppose $Q_Z\cong -E_8 \oplus -\mathbb Z^{n-7}$, and let $\tau = (s, \sigma)$ where $s \in - E_8$, $\sigma\in -\mathbb Z^{n-7}$, and by change of basis we have $0\leq \sigma_0 \leq \sigma_1\leq \ldots \leq \sigma_{n-8}$, and let $\mathfrak c = (\underbrace{0,\ldots,0}_8,\underbrace{1,\ldots,1}_{n-7})$. Then $$0\geq \langle \mathfrak c, \tau\rangle\geq \langle \tau, \tau\rangle=-p.$$ Since $0 = \mathfrak c ^2 + n+1 - 8 \leq -8t_i(K)$ for $i = \frac{\langle\mathfrak c, \tau\rangle + p}{2}$ by Lemma 9, it follows that $t_i(K)=0$, so $i\geq g(K)$, $\langle \mathfrak c, \tau \rangle + p \geq 2g(K)$, and therefore $p>2g(K)-1$.
\end{proof}

\begin{proof}[Proof of Theorem 3]
By Lemma 10, we have an embedding $Q_W\oplus Q_X\hookrightarrow Q_Z = -\mathbb Z^{n+1}$, where the generator of $Q_W$ is sent to some $\tau\in -\mathbb Z^{n+1}$ satisfying $0\leq \tau_0\leq\cdots\leq \tau_n$ and $\langle \tau, \tau\rangle =-(2g(K)-1)$. By Lemma 9, we have that $\mathfrak c^2 + (n+1) -8 \leq -8t_i(K)$ for all $|i|\leq (2g(K)-1)/2$ and $\mathfrak c \in \Char(-\mathbb Z^{n+1})$ such that $\langle\mathfrak c , \tau \rangle + 2g(K)-1 \equiv 2i \ \mod 2(2g(K)-1)$, and, by the sharpness of $X$, for every $|i|\leq (2g(K)-1)/2$ there exists $\mathfrak c$ attaining equality. Note that for any $\mathfrak c \in \{\pm1\}^{n+1}$, we have $\mathfrak c^2 + (n+1) -8 = -8$ and $-8t_i(K)\leq -8$ for all $|i|\leq g(K)-1= \lfloor(2g(K)-1)/2\rfloor$. Let $f(K) = \min\{i\geq 0\ | \ t_i(K)= 1\}$. Then we have the equality
\begin{equation}
    \{\langle \mathfrak c, \tau \rangle \ | \ \mathfrak c \in \{\pm 1\}^{n+1}\} = \{j \in [(2g(K)-1)-2f(K), (2g(K)-1) + 2f(K)]\ | \ j \equiv 1 \ \mod 2\},
\end{equation}
which we rewrite as
\begin{equation}
    \{|\langle \chi, \tau\rangle|\ |\ \chi \in \{0,1\}^{n+1}\} = \{0, 1, \dots, |\tau|_1\},
\end{equation}
from which see that $\tau$ is a changemaker \cite[Lemma 3.2]{Gre15}.

To see that $Q_X = (\tau)^\perp$ on the nose, observe that $\rk \ Q_X = \rk \ (\tau)^\perp$ and $\disc(Q_X) = |H_1(K(2g(K)-1))| = 2g(K)-1 = |\tau| = \disc((\tau)^\perp)$ using \cite[Lemma 3.10]{Gre13} at the last step, so the two lattices coincide.
\end{proof}

\subsection{Linear lattices.}

Let $p>q>0$ be integers. There is a unique continued fraction expansion $$p/q =[x_1,x_2,\ldots,x_n]^-= x_1 - \frac{1}{x_2-\frac{1}{\ddots- \frac{1}{x_n}}}$$
with each $x_i\geq 2$ an integer. The lens space $L(p,q)$ is the oriented boundary of the negative definite 4-manifold $X(p,q)$ given by attaching 4-dimensional 2-handles to a linear chain of $n$ unknots in the boundary of $B^4$, where the framing of the $i^\text{th}$ handle attachment is $-x_i$ (see Figure 2). We note that $X(p,q)$ is sharp, and that $Q_{X(p,q)}$ is indecomposable.

\begin{figure}
   \centering
\begin{tikzpicture}[scale=.95]
\begin{knot}[clip width = 5, flip crossing = 1, flip crossing = 3, flip crossing = 6]
\strand[ultra thick] (0,0) to[curve through={(-1,1)..(0,2)..(1,1)}] (0,0);
\strand[ultra thick] (1.3,0) to[curve through={(.3,1)..(1.3,2)..(2.3,1)}](1.1,0);
\strand[ultra thick]  (2.6 + .5,1-0.86602540378) to[curve through={(2.6,0)..(1.6,1)..(2.6,2)}] (2.6 + 0.5, 1+0.86602540378);
\node at (3.6,1){$\Large \cdots$};
\strand[ultra thick] (3.6+.5, 1-0.86602540378) to[curve through={(4.6,0)..(5.6,1)..(4.6,2)}] (3.6+.5,1+0.86602540378);
\strand[ultra thick] (5.9,0) to[curve through={(4.9,1)..(5.9,2)..(6.9,1)}](5.9,0);

\node at (0,2.3) {$\Large -x_1$};
\node at (1.3,2.3) {$\Large -x_2$};
\node at (2.6, 2.3) {$\Large -x_3$};
\node at (4.6, 2.3) {$\Large -x_{n-1}$};
\node at (5.9,2.3){$\Large -x_n$};

\end{knot}
\end{tikzpicture}
    \caption{A Kirby diagram for $X(p,q)$.}
    \label{fig:kirby_diag_X}
\end{figure}
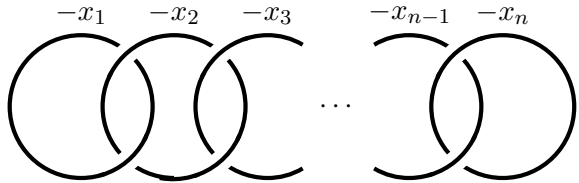

The connected sum of lens spaces $\#_{i=1}^mL(p_i,q_i)$ bounds a canonical sharp 4-manifold $X:=\natural_{i=1}^m X(p_i,q_i)$, a Kirby diagram for which is given by the disjoint union of the surgery diagrams for $X(p_i,q_i)$, $1\leq i \leq m$. The lattice $Q_X \cong \oplus_{i=1}^mQ_{X(p_i,q_i)}$ then contains $m$ indecomposable summands. Improving on our initial appeal to the surgery exact triangle in Heegaard Floer homology, the basic result that a changemaker lattice has at most two indecomposable summands \cite[Lemma 5.1]{Gre13} allows us to conclude the following.

\begin{prop}
If $K(p/q)$ is a connected sum of more than two lens spaces, then $p/q > 2g(K)-1$. 
\end{prop}
\begin{proof}
If $K(p/q)$ is a connected sum of more than two lens spaces, it is an L-space bounding a sharp 4-manifold whose intersection form has more than two indecomposable summands, so if $p/q$ is positive, then $p/q>2g(K)-1$ by combining an appeal to the surgery exact triangle, Theorem 3, and \cite[Lemma 5.1]{Gre13}. If $p/q<0$, we may build a negative definite 4-manifold with boundary $-\mathcal P$---a contradiction since a familiar application of Donaldson's theorem would then imply that $-E_8$ embeds in a diagonal lattice, but it does not. By combining an appeal to the surgery exact triangle and the observation that $|H_1(K(p/q))|=|p|$, we see that $|p/q|>1$, since if $g(K)= 0$, then $K$ bounds a disk and thus $K(p/q)=\mathcal P \# -L(p,q)$. Let $x_i\geq 2$, $1\leq i\leq n$, be integers satisfying $[x_1,\ldots, x_n]^-=|p/q|$. Let $L = K_1 \cup \cdots \cup K_n$ be the framed link in $\mathcal P$ where $K_1 = K$, $K_i$ is a meridian of $K_{i-1}$ for $2\leq i\leq n$, and the framing of $K_i$ is $-x_i$, $1\leq i \leq n$. Denote by $W_{p/q}(L)$ the trace cobordism of surgery on the framed link $L$, which has negative definite intersection form and $\partial W_{p/q}(L)=-\mathcal P \coprod K(p/q)$. Let $X$ be the canonical sharp 4-manifold with boundary a connected sum of lens spaces oriented as $-K(p/q)$ described above. Then $Z:=W_{p/q}(L)\cup_{K(p/q)}X$ is negative definite with boundary $-\mathcal P$. We conclude that $p/q >2g(K)-1$.
\end{proof}

\section{An essential annulus in $\mathcal P_K$.}

With Proposition 11 in hand, if $K(p/q)$ is a connected sum of more than two lens spaces, then we may use the following theorem of Matignon--Sayari to complete the proof of Theorem 1.

\begin{thm}\cite[Theorem 1.1]{MS03}
Let $M$ be an irreducible 3-manifold with boundary a torus $T$. Let $\lambda$ be a slope on $T$ that bounds an orientable surface in $M$ and $g$ the genus of this surface. If there exists a reducing slope $r$, then either $\Delta(r, \lambda)\leq 2g-1$ or else the minimum geometric intersection number of an essential 2-sphere in $M(r)$ with the core of $r$-Dehn filling is 2.\qed
\end{thm}

Suppose that $K\subset \mathcal P$ admits a surgery to a connected sum of more than two lens spaces. For us, $\lambda$ is the 0-framing of $K$, and the slope of this surgery $p/q$ is strictly greater than $2g-1$ by Proposition 11. In particular, $\Delta(0,p/q) = p >2g-1$, so $K(p/q)$ contains an essential 2-sphere which meets the core of $p/q$-Dehn surgery in exactly 2 points by Theorem 12. We use this information to complete the proof of Theorem 1 with the following two lemmas. 

\begin{lem}
Let $p/q > 2g-1$ with $q>1$, and suppose that $K(p/q)$ is reducible for $K\subset \mathcal P$. Then $K$ is an exceptional fiber of $\mathcal P$ and the $(p,q)$-cable of $K$ is a regular fiber.
\end{lem}
\begin{proof}
We invoke Theorem 12 to identify an essential 2-sphere $\hat A\subset K(p/q)$ which intersects the core of the surgery solid torus in two points. Let $A$ denote the essential annulus $\hat A \cap \mathcal P_K$. Then $A$ is separating in $\mathcal P_K$, and $\partial A$ separates $\partial \mathcal P_K$ into two annuli $A_1, A_2$ with $\partial A_i = \partial A$ for $i=1,2$. Observe that each boundary component of $A, A_1,$ and $A_2$ is a $(p,q)$-curve on $\partial \mathcal P_K$. Denote by $T_i$ the torus $A_i \cup A$, $i = 1,2$.

Since $\mathcal P$ is atoroidal, each $T_i$ bounds a solid torus $V_i$ in $\mathcal P$. We will show that in fact each $V_i$ is contained in $\mathcal P_K$, giving a decomposition of $\mathcal P_K$ as $V_1\cup_A V_2$. Then $\mathcal P_K$ admits a Seifert fibering with two exceptional fibers, so $K$ is an exceptional fiber of $\mathcal P$. 

\begin{figure}
    \centering
\begin{tikzpicture}[scale=.8]
\draw[ultra thick, red] (0,0) to[curve through={(-2,2)}] (0,4);
\draw[ultra thick, purple] (0,4) to[curve through={(3,2)}] (0,0);
\draw[ultra thick, blue] (0,4) to[curve through={(2,2)}] (0,0);
\node at (0,2) {\Large $\bullet$};
\node at (0,1.5) {\Large $p/q$};
\node at (-2.5, 2) {\Large $A_1$};
\node at (1.5,2) {\Large $A_2$};
\node at (3.5,2){\Large $A$};
\node at (0,3){\Large $\nu(K)$};
\node at (.5,5){\Large $V_1$};
\end{tikzpicture}

    \caption{The implied Seifert fibering of $V_1$ if $\nu(K) \subset V_1$. Note that $A$ is boundary compressible in $\mathcal P_K$.}
    \label{fig:canthappen}
\end{figure}
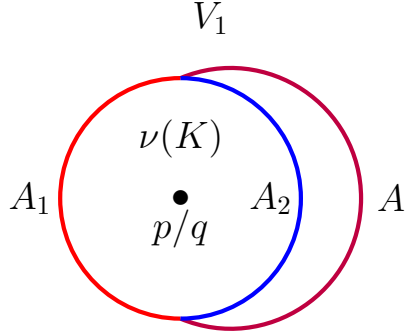

To see that $V_1$ is contained in $\mathcal P_K$, suppose to the contrary that $\nu(K)\subset V_1$. Then $A_2\subset V_1$. Consider the Seifert fibering of $V_1$ over $D^2$ induced by extending the fibering of $A_1$ by $(p,q)$-curves on $\partial \nu(K)$ over $T_1$. In this Seifert fibering, $A_2$ is isotopic to a union of regular fibers and $K$ appears as an exceptional fiber of type $p/q$ (see Figure 3). Since $V_1$ is a solid torus, $K$ is then the only exceptional fiber in this Seifert fibering. It follows that $A$ is isotopic to $A_2$ in $V_1\setminus \nu(K)$, and therefore $A$ is not essential in $\mathcal P_K$.

By the same argument, we see that $V_2$ is also contained in $\mathcal P_K$, hence $\mathcal P_K = V_1 \cup_A V_2$ as desired. Furthermore, since $A$ is not boundary parallel in $\mathcal P_K$, the Seifert fibering of $\mathcal P_K$ induced by those on $V_1$ and $V_2$ has two exceptional fibers (see Figure 4). Equip $\nu(K)$ with the Seifert fibering induced by $(p,q)$-curves on $\partial \nu(K)$ to see that $K$ is an exceptional fiber of $\mathcal P = \mathcal P_K\cup_{A_1 \cup A_2}\nu(K)$, and that the $(p,q)$-cable of $K$ is a regular fiber.\end{proof}

\begin{figure}
    \centering
\begin{tikzpicture}[scale=.8]
\draw[ultra thick, red] (0,0) to[curve through={(-2,2)}] (0,4);
\draw[ultra thick, purple] (0,4)--(0,0);
\draw[ultra thick, blue] (0,4) to[curve through={(2,2)}] (0,0);
\node at (-1,2) {\Large $\bullet$};
\node at (1,2) {\Large $\bullet$};
\node at (-2.5, 2) {\Large $A_1$};
\node at (2.5,2) {\Large $A_2$};
\node at (.35,3.15){\Large $A$};
\node at (0,4.5){\Large $\mathcal P_K$};
\end{tikzpicture}

    \caption{The Seifert fibering of $\mathcal P_K$ induced by the decomposition $\mathcal P_K = V_1\cup_AV_2$ has at least two exceptional fibers since $A$ is not parallel into $\partial \mathcal P_K$. It has at most two exceptional fibers since it is the union of the solid tori $V_1$ and $V_2$ along $A$. The annulus $A$ is a union of regular fibers in the induced Seifert fiberings of both $V_1$ and $V_2$.}
    \label{fig:exceptional}
\end{figure}
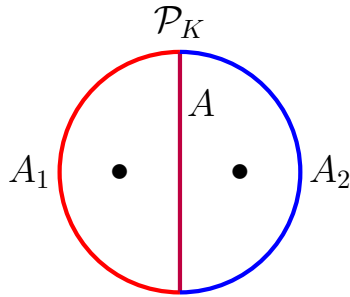

\begin{lem}
Let $p\geq 2g$ be an integral reducing slope for $K \subset \mathcal P$. Suppose that $K(p)$ is a connected sum of more than two lens spaces. Then $K$ is a regular fiber of $\mathcal P$.
\end{lem}
\begin{proof}
We again invoke Theorem 12 to identify an essential properly embedded annulus $A$ in $\mathcal P_K$ with boundary slope $p$. Since $p$ is integral, there is an annulus $A'\subset \nu(K)$ with $\partial A' = \partial A$ and $K \subset A'$. Let $T= A\cup A'$. Since $\mathcal P$ is atoroidal, it follows that $T$ bounds a solid torus. Denote the core of this solid torus by $E$, and observe that $K$ is the $(p,q)$-cable of $E$ for some $p,q\in \mathbb Z$ with $|q|\geq 2$ such that $E(p/q)$ is a connected sum of at least two lens spaces. That $K$ is a regular fiber of $\mathcal P$ then follows from Lemma 13.
\end{proof}

\begin{proof}[Proof of Theorem 1]
Let $K \subset \mathcal P$, and suppose that $K(p/q)$ is a connected sum of more than two lens spaces. By Proposition 11, $p/q> 2g(K)-1$. By Lemma 13, we conclude that $q=1$, since if $p/q>2g(K)-1$ is non-integral and $K(p/q)$ is reducible, then $K(p/q)$ is a connected sum of two lens spaces. By Lemma 14, we conclude that $K = F$. We may furthermore identify that $p/q =|H_1(L(2,1)\#L(3,2)\#L(5,4))|=30$. 
\end{proof}

\begin{rmk}
If we can identify a reducing sphere that meets the core of surgery on $K$ in two points, for $K$ satisfying the hypotheses of Conjecture 5, the then topological arguments in this section can be adapted to prove the conjecture with help from the classification of embedded tori in Seifert fibered spaces \cite[Theorem VI.34]{Jac80}. For a Seifert fibered space over $S^2$ that is not an integer homology sphere L-space, i.e. not $S^3$ or $\pm \mathcal P$ \cite{Eft09}, parts of our argument break down. For example, if $K$ is non-trivial in homology then there need not be a slope on $\partial \nu(K)$ bounding an orientable surface in the exterior of $K$. Even if $K$ is nullhomologous, a positive L-space slope on $K$ need not be bounded below by $2g(K)-1$---in fact, every knot $K$ that is doubly primitive in a Brieskorn sphere that is not an L-space admits an integral surgery slope $p$ to a lens space, and it follows from \cite[Theorem 1]{Ras07} that $|p|\leq 2g(K)-1$.
\end{rmk}


\begin{thebibliography}{99}

\bibitem{Bak14} K. L. Baker: \textit{A cabling conjecture for knots in lens spaces}, Bol. Soc. Mat. Mex. (3) \textbf{20} (2014), no. 2, 449--465.

\bibitem{BO91}M. Boileau and J.-P. Otal: \textit{Scindements de {H}eegaard et groupe des hom\'{e}otopies des petites vari\'{e}t\'{e}s de {S}eifert}, {Invent. Math.} \textbf{106} (1991), no. 1, 85--107.

\bibitem{Eft09}E. Eftekhary: \textit{Seifert fibered homology spheres withtrivial Heegaard Floer homology}, 2009, available at \url{https://arxiv.org/abs/0909.3975}.

\bibitem{Fro96} K. Fr\o yshov: \textit{The Seiberg-Witten equations and four-manifolds with boundary}, Math. Res. Lett. \textbf{3} (1996), no. 3, 373--390.

\bibitem{GAS86} F. Gonz\'alez-Acu\~na and H. Short: \textit{Knot surgery and primeness}, Math.  Proc. Cambridge Philos. Soc. \textbf{99} (1986), no. 1, 89--102.

\bibitem{GL89}C. McA. Gordon and J. Luecke: \textit{Knots are determined by their complements}, J. Amer. Math. Soc. \textbf{2} (1989), no. 2, 371--415.

\bibitem{Gre13}J. E. Greene: \textit{The lens space realization problem}, Ann. of Math. (2) \textbf{177} (2013), no. 2, 449--511.

\bibitem{Gre15}J. E. Greene: \textit{L-space surgeries, genus bounds, and the cabling conjecture}, J. Differential Geom. \textbf{100} (2015), no. 3, 491--506.

\bibitem{Hof98} J. A. Hoffman: \textit{There are no strict great {$x$}-cycles after a reducing or {$P^2$} surgery on a knot}, J. Knot Theory Ramifications \textbf{7} (1998), no. 5, 549--569.

\bibitem{Jac80} W. Jaco: Lectures on three-manifold topology, American Mathematical Society, Providence, R.I. (1980).

\bibitem{Kir10} R. Kirby: Problems in low-dimensional topology (2010). \url{math.berkeley.edu/~kirby/problems.ps.gz}.

\bibitem{Lic62} W. B. R. Lickorish: \textit{A representation of orientable combinatorial {$3$}-manifolds}, Ann. of Math. (2) \textbf{76} (1962), 531--540.

\bibitem{MS03} D. Matignon and N. Sayari: \textit{Longitudinal slope and {D}ehn fillings}, Hiroshima Math. J. \textbf{33} (2003), no. 1, 127--136.

\bibitem{Mos71} L. Moser: \textit{Elementary surgery along a torus knot}, Pacific J. Math. \textbf{38} (1971), 737--745.

\bibitem{OS03}P. Ozsv\'{a}th and Z. Szab\'{o}: \textit{Absolutely graded {F}loer homologies and intersection forms for four-manifolds with boundary}, Adv. Math. \textbf{173} (2003), no. 2, 179--261.

\bibitem{OS04}P. Ozsv\'{a}th and Z. Szab\'{o}: \textit{Holomorphic disks and genus bounds}, Geom. Topol. \textbf{8} (2004), 311--334.

\bibitem{OS05}P. Ozsv\'{a}th and Z. Szab\'{o}: \textit{On knot {F}loer homology and lens space surgeries}, Topology \textbf{44} (2005), no. 6, 1281--1300.

\bibitem{Ras07}J. Rasmussen: \textit{Lens space surgeries and L-space homology spheres}, 2007, available at \url{https://arxiv.org/abs/0710.2531}.

\bibitem{Sav12}N. Saveliev: Lectures on the topology of 3-manifolds, 2nd edition, Walter de Gruyter \& Co., Berlin (2012).

\bibitem{Tan09}M. Tange: \textit{Lens spaces given from {$L$}-space homology 3-spheres}, Experiment. Math. \textbf{18} (2009), no. 3, 285--301.

\bibitem{Wal60}A. H. Wallace: \textit{Modifications and cobounding manifolds}, Canadian J. Math. \textbf {12} (1960), 503--528.
\end{thebibliography}
\end{document}